%% file: B4NThType.tex
\documentclass{amsart}
\usepackage{amsfonts,amsmath,amssymb,amscd,latexsym,graphicx,epsfig,color,slashbox}

\title{Fast algorithmic Nielsen-Thurston classification of four-strand braids}
\theoremstyle{plain}

\newtheorem{definition}{Definition}[section]
\newtheorem{proposition}[definition]{Proposition}
\newtheorem{corollary}[definition]{Corollary}
\newtheorem{lemma}[definition]{Lemma}
\newtheorem{theorem}[definition]{Theorem}
\newtheorem{claim}{Claim}
\newtheorem{example}[definition]{Example}
\newtheorem{remark}[definition]{Remark}
\newtheorem{conjecture}[definition]{Conjecture}
\begin{document}

\author{Matthieu Calvez}
\address{Matthieu Calvez, IRMAR (UMR 6625 du CNRS), Universit\'e de Rennes 1,
Campus de Beaulieu, 35042 Rennes Cedex, France}
\email{matthieu.calvez@univ-rennes1.fr}

\author{Bert Wiest}
\address{Bert Wiest, IRMAR (UMR 6625 du CNRS), Universit\'e de Rennes 1,
Campus de Beaulieu, 35042 Rennes Cedex, France}
\email{bertold.wiest@univ-rennes1.fr}
\subjclass[2010]{20F36, 20F10, 20F65}

\begin{abstract}
We give an algorithm which decides the Nielsen-Thurston type of a given 
four-strand braid. The complexity of our algorithm is quadratic with 
respect to word length. The proof of its validity is based on a result 
which states that for a reducible 4-braid which is as short as possible 
within its conjugacy class (short in the sense of Garside), reducing curves
surrounding three punctures must be round or almost round. 
As an application, we give a polynomial time solution to the conjugacy
problem for non-pseudo-Anosov four-strand braids.
\end{abstract}


\maketitle

\section{Introduction}
\subsection{Statement of the main result}
In the 1980's Thurston gave a complete classification of the elements of the mapping class groups of surfaces into three types: periodic, pseudo-Anosov (pA), or reducible.

During the 1990's, algorithms which decide the type (called the Nielsen-Thurston type) of a given mapping-class were constructed via the theory of train-tracks (\cite{bestvinahandel}, \cite{los}). Unfortunately, the complexity of these algorithms remains unknown. Even in the particular case of the $n$-strand braid group $B_n$ (i.e.\ the mapping class group of an $n$-times punctured disk) the problem of deciding whether a given braid is reducible or not (which we call \emph{reducibility problem}) has currently no known polynomial time solution.

An alternative algorithm in this particular case, using Garside theory, is given in the paper \cite{bertjuan} (which builds on \cite{leelee}). However, the complexity of the algorithm in~\cite{bertjuan}, while conjectured to be polynomial, strongly depends on an open question. 
In the present paper we give a polynomial solution to the reducibility problem in the particular case of 4-braids. More precisely we establish the following result:
\begin{theorem}
\label{main}
There is an algorithm which decides the Nielsen-Thurston type of any given 4-braid $x$, and whose running time is $O(\mathfrak l^2)$, where $\mathfrak l$ denotes the length of $x$ in the classical Artin generators $\sigma_i$.
\end{theorem}

\begin{corollary}\label{C:ConjugPbm}
There is a polynomial time solution to the conjugacy search problem for 4-strand braids in the non-pseudo-Anosov case, i.e.~deciding whether two braids, at least one of which is not pA, are conjugate, and if they are, finding a conjugating element.
\end{corollary}



Together with Theorem 2 of \cite{ito} we also obtain the following:
\begin{corollary}\label{ito} 
There is an algorithm with the following properties. It takes as its input
a 4-strand braid~$x$ whose Dehornoy floor is at least 3, and whose closure 
is a knot~$K$; and it outputs, after a calculation 
in time $O(($length$(x))^2)$, whether~$K$ is
\begin{itemize}
\item a torus knot
\item a satellite knot (and in this case it also outputs the reducing
torus)
\item a hyperbolic knot.
\end{itemize}
\end{corollary}

The plan of the paper is as follows.
In this first section we recall some facts about reducible braids and Garside theory; the second section is devoted to the proof of Theorem \ref{main}
and Corollary \ref{C:ConjugPbm},
modulo the key technical result (Proposition \ref{3punctures}) whose proof is deferred to the third section. Finally in the fourth section we give some examples and conjectures related to the reducibility problem in braid groups.

\subsection{Reducible braids}
Let $D_n$ be the closed disk in $\mathbb{C}$ with diameter $[0,n+1]$ and with the points $\{1,\cdots,n\}$ removed. It is known \cite{farbmargprimer,juancourse} that the $n$-strand braid group $B_n$ is identified with the mapping class group of $D_n$. 

Hence there is a (right) action of the braid group on the set of 
isotopy classes of simple closed curves in~$D_n$. By abuse of notation we do not distinguish between a simple closed curve and its isotopy class. We denote the curve resulting from the action of the braid~$x$ on the curve~$\mathcal C$ by $\mathcal C*x$.
A simple closed curve is said to be nondegenerate if it surrounds more than one puncture and less than~$n$.

A braid $x$ is said to be \emph{reducible} if it preserves setwise a family of nondegenerate simple closed curves; such a curve is then called a reduction curve for~$x$. A reduction curve of $x$ is said to be \emph{essential} if it does not cross any other reduction curve. 
The set of all essential reduction curves of $x$ is called the \emph{canonical reduction system} of~$x$ and denoted by $CRS(x)$. It is known that the set $CRS(x)$ is non-empty if and only if $x$ is reducible nonperiodic (see \cite{birman}).

A braid $x$ is said to be \emph{periodic} if some power of $x$ is a power of the full twist~$\Delta^2$. Pure periodic braids are known to be powers of $\Delta^2$
(where $\Delta$ is the half-twist of all strands, defined in Artin's generators by the formula $$\Delta= 
(\sigma_1\ldots\sigma_{n-1})(\sigma_1\ldots\sigma_{n-2})\ldots(\sigma_1\sigma_2)\sigma_1.)$$

In what follows we will take ``reducible" to mean ``reducible nonperiodic". Note that a braid $x$ is reducible if and only if every power $x^t$ (with $t\neq0$) of $ x$ is reducible. Note also that reducibility is a property invariant under conjugation.
 
The following definition, which comes from \cite{bertcomplexity}, uses the notion of canonical length of a braid, which will be recalled in the next subsection.
\begin{definition}\label{D:round}
We say that a simple closed curve in $D_n$ is round if it is homotopic to a geometric circle.
The complexity of a simple closed curve $\mathcal C$ in $D_n$ is defined to be the smallest canonical length of a positive braid which sends $\mathcal C$ to a round curve. 
(Note that if some positive braid sends $\mathcal C$ to a round curve, then there is another positive braid of the same canonical length sending the round curve back to $\mathcal C$.)
\end{definition}
Hence the curves of complexity 0 are the round curves; we shall call \emph{almost-round} the curves of complexity 1.  In Figure \ref{toto} are represented three simple closed curves in $D_4$: the first is round, the second is almost-round and is sent to a round curve by the permutation braid $\sigma_1\sigma_3$, the third is of complexity  2 and is obtained from the first by applying $\sigma_2^{-2}$.

\begin{figure}[htb]
\centerline{\input{figurecourbered.pstex_t}}
\caption{}\label{toto}
\end{figure}

Let us introduce a notion which is closely related to Definition \ref{D:round} (see ~\cite{leelee}): the minimal standardizer of a family~$\mathcal C$ of disjoint simple closed curves in $D_n$ is defined to be the smallest positive braid (for the prefix order on $B_n$) which sends~$\mathcal C$ to a family of round curves. Hence the complexity of a curve $\mathcal C$ coincides with the canonical length of its minimal standardizer. In particular, we may consider the minimal standardizer of the canonical reduction system of a reducible braid. 

\subsection{Garside theory}\label{garsidesection}
For an introduction to the classical Garside structure on braid groups, the reader is referred to the papers \cite{garside} and \cite{elrifaimorton} where the notion of left normal form is defined. The second one also introduces the notions of $\inf$, $\sup$ and canonical length. Recall that if a braid $x=\Delta^px_1\cdots x_r$ is in left normal form (where the $x_i$'s are positive permutation braids) then $\sup(x)=p+r$, $\inf(x)=p$ and the canonical length of $x$, denoted by $\ell(x)$, is defined by $\ell(x)=r$.

We also recall that the super summit set $SSS(x)$ of a braid~$x$ is the (finite, nonempty) subset of the conjugacy class of~$x$ containing all those elements which have both minimal $\sup$ and maximal $\inf$ (or equivalently minimal canonical length)
within their conjugacy class \cite{elrifaimorton}. The minimal value of $\ell$ within the conjugacy class
of a braid $x$ is called summit canonical length of $x$ and is denoted
$\ell_s(x)$.

In order to be able to construct at least one element of $SSS(x)$ for any given braid~$x$, let us recall a special kind of conjugation, the so-called \emph{cyclic sliding} \cite{gebhardtjuan}. Let us denote by $\tau$ the conjugation by the half-twist $\Delta$, by $\wedge$ the gcd associated to the prefix order and, for a simple element~$z$ of $B_n$, by $\partial(z)$ the right complement to $\Delta$, that is $\partial(z)=z^{-1}\Delta$.

\begin{definition}[\cite{gebhardtjuan}] Let $x=\Delta^{p}x_1\cdots x_r$ be a braid in left normal form, with $r>0$; one defines the cyclic sliding, denoted $\mathfrak{s}$, by the following conjugation:
$$\mathfrak{s}(x)={\mathfrak{p}(x)}^{-1}x{\mathfrak{p}(x)} \text{ \ where }
\mathfrak p(x)=\tau^{-p}(x_1)\wedge \partial(x_r).$$
The simple braid $\mathfrak{p}(x)$ is called the preferred
prefix of $x$. (If $\ell(x)=0$, one also can define cyclic sliding: in this case $x$ is some power of $\Delta$ and $\mathfrak s$ is the trivial conjugation).
\end{definition}

In \cite{gebhardtjuan} it was shown that, starting from $x$ and iteratively applying the cyclic sliding operation $(\ell(x)-1)\cdot (\frac{n(n-1)}{2}-1)$ times yields an element of $SSS(x)$.
(This is analogue to previous results of \cite{elrifaimorton} and \cite{bkl}
concerning the so-called cycling and decycling operations.)
This leads to the following result:
\begin{theorem}[\cite{gebhardtjuan}]
\label{bkl}
There is an algorithm which takes as its input a braid~$x$ with~$n$ strands, runs for time $O\left(\ell(x)^2 n^2\right)$, and outputs an element~$x'$ of $SSS(x)$ which is of the form $x'=\mathfrak{s}^k(x)$ for some 
$k\in\mathbb{N}$.
\end{theorem} 


The cyclic sliding operation behaves well with respect to the reducibility problem:
\begin{proposition}[\cite{juan2010}]\label{bernardete}
Let $x=\Delta^{p}x_1\cdots x_r$ be a braid in left normal form and let~$\mathcal C$ be a round simple closed curve such that $\mathcal C*x=\mathcal C$. Then the curve $\mathcal C*\mathfrak p(x)$ is also round. In particular, the braid $\mathfrak{s}(x)$ also preserves a round curve.
\end{proposition}  

Now every reducible braid has some conjugate with
round essential reduction curves. Thus Proposition~\ref{bernardete}  
together with Theorem~\ref{bkl} yields:
\begin{corollary}[\cite{bernardete}]\label{C:RoundInSomeSSS} 
For any reducible braid $x$ there exists an element of $SSS(x)$ with a round essential reduction curve.
\end{corollary}

We shall also need the following result which essentially comes from~\cite{gebhardt}.

\begin{proposition}\label{gebhardt}
Let $x=\Delta^px_1\cdots x_r$ be the left normal form of a reducible braid which preserves a simple closed curve $\mathcal C$ of complexity $s$. Then the complexity of the curve $\mathcal C*\Delta^px_1\cdots x_i$ is bounded above by $s$ for every $i$ with $1\leqslant i \leqslant k$.
\end{proposition}

\begin{proof}
Notice that the case $s=0$ in the proposition is proved in 
Theorem 5.7 in \cite{bernardete}.
Let $P$ be the minimal standardizer of $CRS(x)$. By Theorem 4.9 in \cite{leelee}, we have:
$\inf(P^{-1}xP)\geqslant\inf(x)$ and $\sup(P^{-1}xP)\leqslant\sup(x)$. 
Under some minor changes, Proposition 2.1 and Corollary 2.2 in \cite{gebhardt} assert the existence of positive braids $P_0=\tau^p(P),P_1,\ldots,P_r$ such that $$P^{-1}xP\ \text{has left normal} \ \Delta^p(P_0^{-1}x_1P_1)\ldots(P_{r-1}^{-1}x_rP_r)$$ (with possibly some indices $0\leqslant i'<i''\leqslant r-1$ such that 
$u_i^{-1}x_{i+1}u_{i+1}=\Delta$ for $i\leqslant i'$ and $u_i^{-1}x_{i+1}u_{i+1}=1$ for $i\geqslant i''$) and $$\inf(P_i)\geqslant \inf(P),\sup(P_i)\leqslant \sup(P)\ \text{for}\ i=1\ldots r.$$
The case $s=0$ of the current Proposition yields the roundness of the curves 
$$(\mathcal C*P)*(\Delta^p(P_0^{-1}x_1P_1)\ldots(P_{i-1}^{-1}x_{i}P_{i}))$$
for all $i=1,\ldots,r$.
The latter can now be rewritten as $(\mathcal C*(\Delta^px_1\ldots x_i))*P_i$. The complexity of the curve 
$\mathcal C*\Delta^px_1\ldots x_i$ is thus at most $\ell(P_i)\leqslant \ell(P)=s$. This shows Proposition~\ref{gebhardt}.
\end{proof}

The special case $s=0$ of Proposition~\ref{gebhardt} together with 
Corollary~\ref{C:RoundInSomeSSS}, means that for \emph{some} (but not necessarily every) element of the $SSS$ of a reducible braid, reducibility is easy to detect.

However this is not sufficient in order to obtain a polynomial-time algorithm for detecting reducibility since the size of the $SSS$ may grow exponentially with both word length and braid index (see \cite{juan2}). This difficulty does not yet appear in the case $n=3$, as will be shown in Proposition \ref{P:reducibilityb3}.
Our strategy for proving Theorem~\ref{main} is to show that also in the case of four strand braids, the reducibility is easy to see for \emph{every} element in the SSS. The following is our main technical result:

\begin{proposition}
\label{3punctures}
Let $x$ be a reducible 4-braid such that $x\in SSS(x)$. Suppose that $x$ has an essential reduction curve surrounding 3 punctures. Then this curve is round or almost round.
\end{proposition}

Notice that an essential reduction curve of a 4-braid~$x$ surrounding three punctures is fixed by $x$ (not sent to another reduction curve).
Now, let us recall from~\cite{gebhardtjuan} a last Garside-theoretical notion:

\begin{definition}[\cite{gebhardtjuan}]
Let $x=\Delta^px_1\cdots x_r$ be a braid in left normal form. We say that~$x$ is rigid if its preferred prefix is trivial, i.e.\ if the equality $\mathfrak p(x)=1$ holds. 
That means that the pair $x_r\tau^{-p}(x_1)$ is left-weighted.  
\end{definition}

We end this section with some results specifically concerning 3-braids.
We start by recalling that in this case there are only four simple braids other than $1$ and~$\Delta$ (i.e.\ nontrivial strict prefixes of~$\Delta$), namely $\sigma_1$, $\sigma_2$, $\sigma_1\sigma_2$ and $\sigma_2\sigma_1$.

\begin{proposition}
\label{formenormale3}
Let $x$ be a 3-braid with $\inf(x)=p$ and $\ell(x)=r$. Let ${x_1,\ldots, x_r \in B_3^{+}}$ be non-trivial simple elements different from $\Delta$. 
If $x=\Delta^px_1\cdots x_r$, then this is the left normal form of $x$.
\end{proposition}
\begin{proof}
If it was not, making the pair $x_ix_{i+1}$ left-weighted (for some $1\leqslant i\leqslant r-1$) either would create one factor $\Delta$ (contradicting $\inf(x)=p$) or would decrease the number of factors (contradicting $\sup(x)=r+p$), or both.
\end{proof}

\begin{proposition}
\label{rigid}
Let $x$ be a 3-braid with $\ell_s(x)>1$. Then $SSS(x)$ is the set of rigid conjugates of $x$.
\end{proposition}
\begin{proof}
Let $y$ be an element of $SSS(x)$, written $y=\Delta^py_1\ldots y_r$ in left normal form, and suppose that $y$ is not rigid. 
Consider the conjugate $$z=(\tau^{-p}(y_1))^{-1}y\tau^{-p}(y_1)=\Delta^p y_2\ldots y_{r}\tau^{-p}(y_1)$$ of $y$. As in the proof of Proposition~\ref{formenormale3}, the non left-weightedness of the pair $y_r\tau^{-p}(y_1)$ implies that $z$ has smaller canonical length than 
$y$. This is a contradiction because $y\in SSS(x)$. 

Conversely, according to Theorem \ref{bkl}, a rigid braid $y$ conjugate to $x$ belongs to $SSS(x)$ since $\mathfrak s(y)=y$.
\end{proof}

\begin{proposition}
\label{P:reducibilityb3}
The reducibility problem as well as the conjugacy search problem can be solved in polynomial time for three-strands braids.
\end{proposition}

\begin{proof}
We claim that in the particular case of 3-strand braids, the size of the Super Summit Set of a braid $x$
is bounded above by $2\cdot\ell_s(x)$. Thus computing the whole Super Summit Set of any 3-braid is doable in polynomial time according 
to Proposition 6.2 in \cite{francomeneses} and consequently, the 
reducibility problem and the conjugacy search problem can be solved quickly.

Let us prove the claim we made. 
First, the $SSS$ of elements of canonical length~0 (i.e.\ powers of $\Delta$) contain only one element. Then, any braid of canonical length~1
is an element of its Super Summit Set since it is not conjugate to a power of $\Delta$ and thus the canonical length
is already minimal within the conjugacy class. Thus as both letters $\sigma_1$ and $\sigma_2$ are conjugate to each other (by $\Delta$)
and because $\inf$ has constant value in the Super Summit Set
one has $$SSS(\Delta^k\sigma_1)=\{\Delta^k\sigma_1,\Delta^k\sigma_2\},$$
$$SSS(\Delta^k\sigma_1\sigma_2)=\{\Delta^k\sigma_1\sigma_2,\Delta^k\sigma_2\sigma_1\},$$ for any $k\in \mathbb Z$. This proves the claim
for braids of canonical length~1.

Suppose then that we are given a 3-braid $x$ of summit canonical length $\ell>1$ together with an element $y=\Delta^py_1\ldots y_{\ell}$ (in left normal form) of $SSS(x)$.
Recall the cycling operation, which consists, for a general braid $u=\Delta^pu_1\ldots u_r$ in left normal form, of 
the conjugation $\mathbf c(u)=
(\tau^{-p}(u_1))^{-1}u\tau^{-p}(u_1)$ 
\cite{elrifaimorton}.
Applied to our braid~$y$, we have $\mathbf c(y)=\Delta^py_2\ldots y_{\ell}\tau^{-p}(y_1)$, and since $y$ is rigid (by Proposition~\ref{rigid}), $c(y)$ is in left normal form as written and rigid.
Notice that~$\tau(y)$ as well belongs to $SSS(x)$. Now consider the first $\ell$ cyclings of both~$y$ and~$\tau(y)$. 
This produces at most $2.\ell$ elements of $SSS(x)$. More precisely
we have $\mathbf c^{\ell}(y)=\tau^{-p}(y)$, which is~$y$ if $p$ is even and $\tau(y)$ if $p$ is odd (recall that~$\Delta^2$ is central). Thus we have just computed either two closed orbits under cycling, conjugate to each other by $\Delta$, or one closed orbit under cycling, self conjugated by~$\Delta$. We thus have a subset $\mathcal O$ of $SSS(x)$ closed by cycling and conjugation by $\Delta$. By Proposition 4.14 in \cite{francomeneses}, if the inclusion $\mathcal O\subset SSS(x)$ is strict, then there must exist another element
$v \in SSS(x)$ which can be obtained by conjugation of some 
$z=\Delta^pz_1\ldots z_{\ell}\in \mathcal O$ by a minimal conjugator~$s$. This minimal conjugator is a simple, and is shown in in the proof of Corollary 2.7 in \cite{bggm2} to be either a prefix of $\tau^{-p}(z_1)$ or of $\partial(z_r)$. 
However, for the rigid 3-braid $z$, conjugating by a strict prefix of either
$\tau^{-p}(z_1)$ or of $\partial (z_\ell)$ would increase the canonical length. Thus $s$ must actually be the whole factor: 
either $s=\tau^{-p}(z_1)$ (in which case $v=\mathbf c(z)$) or $s=\partial(z_{\ell})$. 
In the latter case, 
$$\partial(z_{\ell})^{-1}z\partial(z_{\ell})=
\Delta^p\ \tau^{p+1}(z_\ell) \tau(z_1) \ldots \tau(z_{\ell-1})=
\mathbf c^{\ell-1}(\tau^{p+1}(z)).$$ 
This shows that we already had all the elements of $SSS(x)$ in $\mathcal O$. 
Since $\mathcal O$ has at most $2.\ell$ elements, the claim is proven.
\end{proof}


\section{Proposition \ref{3punctures} implies Theorem \ref{main}}
In this section we prove Theorem \ref{main} with the aid of Proposition \ref{3punctures}.
The first observation is that there are only 2 round curves and 4 almost-round curves surrounding~3 punctures in $D_4$ (see Figure \ref{rond}).  
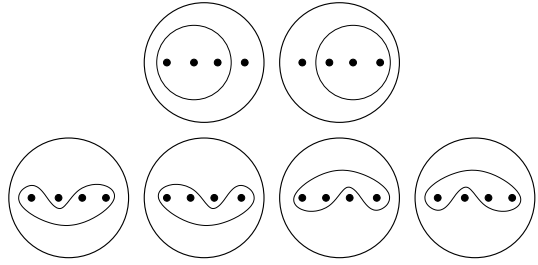
\begin{figure}[htb]

\centerline{\input{rond.pstex_t}}
\caption{Above, the round curves surrounding 3 punctures in $D_4$; below, the set of curves obtained from the above ones by applying a simple braid; i.e. the set of almost-round curves surrounding 3 punctures.}\label{rond}
\end{figure}

\begin{lemma}\label{L:SphereDiskTransl}
There is an algorithm which decides in time $O({\mathfrak l}^2)$ whether a 4-braid~$x$ given as a word of length $\mathfrak l$ in Artin's generators, is reducible with an essential reduction curve surrounding 3 punctures.
\end{lemma}
\begin{proof}
We recall that $\ell(x)\leqslant\mathfrak l(x)\leqslant6\ell(x)$. Now according to Theorem \ref{bkl}, an element~$y$ of $SSS(x)$ can be computed in time $O(\mathfrak l^2)$. According to Proposition~\ref{3punctures},~$x$ admits an essential reduction curve surrounding 3 punctures if and only if $y$ 
preserves one of the six curves in Figure \ref{rond}. The check whether this is the case 
can be performed in time $O(\mathfrak l)$: according to Proposition \ref{gebhardt}, it is sufficient to test whether the images of these curves under each successive Garside factor of $y$ are still 
round or almost round.
\end{proof}

In order to overcome the fact that we don't have any analogue of Proposition~\ref{3punctures} for curves surrounding 2 punctures, we will consider 4-braids as mapping classes of a~${\text{5-punctured}}$ sphere by collapsing the boundary of $D_4$. The five punctures lie on the equator and we shall number them from 1 to 5, the fifth being the new puncture, on the far side of the sphere.

In the following construction $x$ will be supposed to be pure; this assumption is fullfilled up to taking a power 2, 3 or 4. Then for $j=1,\ldots ,4$, blowing up the $j$th puncture to become the boundary of a new 4-punctured disk $D_4$ yields from $x$ a new 4-braid which we will denote by $\tilde{x}_j$.

\newcommand{\xj}{\tilde{x}_j}
\begin{lemma}\label{sphere}
For $j=1,\ldots, 4$, the length $\mathfrak l(\xj)$ of $\xj$ in Artin's generators is bounded above by $3\mathfrak l(x)$.
\end{lemma}
\begin{proof}
We show that each letter of $x$ gives rise to no more than 3 letters. Let us describe in detail how the first letter of $x$ is transformed. In order to identify 
the $j$th puncture with the new boundary  we make a rotation of the equator bringing 
the $j$th puncture behind the sphere at the place of the fifth, and then we renumber the punctures following the rule:
\begin{eqnarray*}
i& \leadsto & \ {i-j+5}\,\ \ \textrm{if} \,\ \ \,i\leqslant j,\\
i& \leadsto & \ {i-j} \,\ \ \ \ \ \ \ \textrm{if}\ \ \ i>j.
\end{eqnarray*}

If the first letter of $x$ is $\sigma_i$, for $1\leqslant i\leqslant j-2$ or for $j+1\leqslant i\leqslant n-1$, then its image in $\tilde{x}_j$ is easy to compute: it is $\sigma_{i-j+5}$ or~$\sigma_{i-j}$, respectively. If $j\geqslant 2$ and the first letter of $x$ is $\sigma_{j-1}$, then this corresponds to a move of the puncture numbered~${j-1}$ which goes to the right above the puncture numbered~$j$. After our rotation the corresponding move involves the fourth puncture which goes above the other punctures to the first position. This corresponds to the braid~${\sigma_3^{-1}\sigma_2^{-1}\sigma_1^{-1}}$. 
In a similar way we can compute the images of all Artin's generators:
\begin{eqnarray*}
\sigma_i^{\pm1} & \leadsto & \sigma_{i-j+5}^{\pm1}\,\ \ \textrm{if} \,\ \ \,i<j-1,\\
\sigma_i^{\pm 1} & \leadsto & \sigma_{i-j}^{\pm 1} \ \ \ \ \  \ \textrm{if}\ \ \ i>j,\\
\sigma_{j-1}^{\pm 1} & \leadsto &\sigma_3^{\mp 1}\sigma_2^{\mp 1}\sigma_1^{\mp 1} \ \ \ \ (\text{for $j\geqslant 2$}),\\
\sigma_j^{\pm 1} & \leadsto & \sigma_1^{\mp 1}\sigma_2^{\mp 1}\sigma_3^{\mp 1} \ \ \ \ (\text{for $j\leqslant 3$}).
\end{eqnarray*}
Notice that the first letter of $x$ induces a permutation of the punctures which possibly sends the puncture numbered $j$ to another position. Thus computing the image of the second letter of $x$ with the aid of the above formulae requires a renumbering of the punctures, according to the permutation involved. The images of the following letters of $x$ are computed in the same way. 
\end{proof}

\begin{lemma}
Let $x$ be a reducible braid without any essential reduction curve surrounding 3 punctures. Then there exists $j$ between 1 and 4 such that $\xj$ is reducible with an essential reduction curve surrounding 3 punctures.
\end{lemma}
\begin{proof}
Note that the reducibility of $x$ is equivalent to the reducibility of each $\tilde{x}_j$ for~${j=1,\ldots,4}$. Now under the assumption of the lemma, $x$ admits an essential reduction curve surrounding 2 punctures. After collapsing the boundary of $D_4$ this curve divides the sphere into two connected components, one with 2 punctures, the other with 3 punctures. Blowing up one of the first two to the new boundary achieves the proof of the lemma.
\end{proof}

To conclude the proof of Theorem \ref{main}, we can give the algorithm which solves the reducibility problem in $B_4$.
\begin{itemize}
\item [0]INPUT: A braid word $x$ in the letters $\sigma_i^{\pm 1}$.
\item[1] Compute a pure power $x^t$ of $x$ and do $x:=x^t$.
\item [2] Compute the left normal form of $x$. 
\item [3] Test whether $x$ is periodic (that is test if $x$ is a power of $\Delta^2$). If yes, then RETURN ``$x$ periodic" and STOP. Else go to 4.
\item [4] Apply iterated cyclic sliding to $x$ until the canonical length has not decreased during the last five iterations.

\item [5]For the element of $SSS(x)$ obtained at the previous step, test whether it has a round or almost round essential reduction curve surrounding 3 punctures. If yes, RETURN ``$x$ reducible" and STOP. If not go to the following step.
\item [6]For $j=1,\ldots,4$, compute $\tilde{x}_j$, apply to it steps 2 and 4. Test whether the element of $SSS({\tilde{x}}_j)$ thus obtained has a round or almost round essential reduction curve surrounding 3 punctures; if the answer is positive for some~$j$, RETURN ``$x$~reducible" and STOP. If the answer is negative for all $j$, RETURN ``$x$ pseudo-Anosov" and STOP.
\end{itemize}

\begin{proof}
We now prove Corollary~\ref{C:ConjugPbm}. Given two braid words
$x^{(1)}$ and $x^{(2)}$ of length~$\mathfrak l$ in the generators 
$\sigma_i^{\pm 1}$ ($i=1,2,3$), we want to test 
whether they are conjugate, and we want to do so in time $O(P(\mathfrak l))$, 
where $P$ is a polynomial. In order to achieve this, we first check whether 
they are periodic. If one of them is 
and the other one isn't, then they are not conjugate; if they both are,
we are able to solve the conjugacy search problem for them in polynomial time, see \cite{juan2}.

Next we want to check whether both braids admit a reducing curve 
surrounding three punctures. In order to do so, we calculate (using iterated
cyclic sliding) braids~$P_1$ and $P_2$ such that 
$P_i^{-1}x^{(i)} P_i\in SSS(x^{(i)})$ (see Theorem~\ref{bkl}). For each of these conjugated braids 
we can 
decide whether they admit a
round or almost round reducing curve surrounding~3 punctures (see Lemma \ref{L:SphereDiskTransl}). If one of them has and the other one
has not then they are not conjugate. If they both have then, possibly 
after adding one factor to $P_i$, we can even assume that
$P_i^{-1}x^{(i)} P_i$ (for $i=1,2$)
both have \emph{round} reducing curves. These two braids are
now conjugate if and only if two conditions are satisfied. Firstly the
winding numbers of the outer strand with the three inner strands must
coincide, and secondly the inner 3-braids must be conjugate. Both
of these conditions can be checked very quickly (see Proposition~\ref{P:reducibilityb3}). 

The most difficult case occurs when neither of the braids admits a
reducing curve surrounding three punctures, and we have to check for 
reducing curves surrounding two punctures. If a braid $x^{(i)}$ has
a canonical reducing curve surrounding two punctures, then a conjugate 
of $x^{(i)}$ in which the reducing curve is round
can be found as follows. We consider the twelfth power $y^{(i)}=x^{(i)\,12}$
which is guaranteed to be pure and which has the same canonical reduction
curves as $x^{(i)}$. One of the braids $\tilde{y}^{(i)}_j$ ($j=1,\ldots,4$)
has a canonical reduction curve surrounding three punctures. As in the
previous paragraph, we can find a braid $\tilde{P}_i$ so that 
$\tilde{P}_i^{-1}\tilde{y}^{(i)}_j\tilde{P}_i$ has a \emph{round} canonical
reduction curve surrounding three punctures. The corresponding homeomorphism
of the five times punctured sphere also has a reduction curve intersecting
the equator only twice. Thus as in the proof of Lemma~\ref{sphere} we can 
explicitly write a 4-strand braid $P_i$ such that the pure braid
$P_i^{-1} y^{(i)} P_i$ has a round reducing curve surrounding two punctures.
Thus the twelfth root $P_i^{-1} x^{(i)} P_i$ has the same round reducing
curves.

To summarize, if the search for round reducing curves surrounding two
punctures is unsuccessful in both braids $x^{(i)}$ ($i=1,2$), then 
both braids are pseudo-Anosov, and we cannot answer the question
whether they are conjugate. If for one of them the search is successful
and for the other it is not, then the two braids are not conjugate. 
If for both braids we find conjugates with round reducing curves
surrounding two punctures, then we consider for each of them a 3-braid 
$z^{(i)}$ obtained by merging the two inner strands of a reducing curve 
into a single strand. We then solve the reducibility problem for each of 
the $z^{(i)}$'s. This can be done 
quickly, according to Proposition \ref{P:reducibilityb3}. 
Our two braids $x^{(i)}$ are conjugate only if
the braids $z^{(i)}$ have the same Nielsen-Thurston type. 

The braids $z^{(i)}$ cannot be periodic because otherwise the braids $x^{(i)}$ would have a reducing curve
surrounding 3 punctures. 
If the $z^{(i)}$'s are pA, then the braids $x^{(i)}$ are conjugate if and only if 
the following two conditions are satisfied. Firstly, the winding
numbers of the two inner strands must coincide, and secondly, the two
three-strand braids $z^{(i)}$ must be conjugate. Again, both of these conditions can be
checked in polynomial time. 
Finally if both braids $z^{(i)}$ are reducible, one can find 4-braids $u^{(i)}$ conjugate to~$x^{(i)}$ so that 
$u^{(i)}$ has 2 round reducing curves, each of them surrounding two punctures. Moreover, 
the conjugating element can be explicitly written as a product of
simple 4-braids corresponding to the conjugating 3-braids occurring during iterated cycling applied to $z^{(i)}$, but with the strand corresponding to the round reducing curve of 
$P_i^{-1}x^{(i)}P_i$ duplicated.
Now,  the braids $x^{(i)}$ are conjugate if and only if 
the winding number of the two fat strands in
$u^{(1)}$ is the same as in~$u^{(2)}$, and 
if the two winding numbers of the pairs of strands inside the two tubes of $u^{(1)}$ coincide with those in~$u^{(2)}$. This can once again be checked quickly.
\end{proof}


\section{Proof of Proposition \ref{3punctures}}

\subsection{Outline of the proof, notation}

We shall give a proof by contradiction. So let us suppose that there exists a reducible 4-braid in its own SSS which admits an essential reduction curve of complexity 2 or greater than 2 surrounding 3 punctures. By multiplying the braid by a sufficiently high power of $\Delta^2$ one may suppose that it is positive. Moreover we may suppose this essential reduction curve to be of complexity exactly~2. If it was greater, then the ``convexity" of the SSS (see Corollary 4.2. of~\cite{elrifaimorton}) and the existence in the SSS of a braid with round essential reduction curve would yield another braid in the SSS with a complexity 2 essential reduction curve.

Hence we start with a positive reducible 4-braid $x$ in its own SSS with an essential reduction curve of complexity 2 surrounding 3 punctures. Let us denote this curve by $\mathcal C$.
We recall that $\mathcal C$ is fixed by $x$, because it is essential.
We call \emph{outer} the strand whose puncture base is not surrounded by the curve $\mathcal C$, the other three are \emph{inner}.
For the rest of this section we denote by~$\hat x$ the 3-braid $z^{(i)}$ obtained from~$x$ by removing the outer strand.

The plan of the proof is as follows: first we list all curves of complexity 2 surrounding~3 punctures in $D_4$. Next we prove that $x$ and $\hat x$ have the same canonical length, and that $\hat x$ is a rigid braid. 
Finally, using a careful case-by-case analysis, we show that Proposition \ref{gebhardt} and the rigidity of $\hat x$ together imply that none of the curves in our list can be an essential  reduction curve for $x$.

\subsection{Curves of complexity 2 surrounding 3 punctures in $D_4$}

Let us classify the simple closed curves of complexity 2 surrounding 3 punctures in~$D_4$. In order to describe them we may also consider isotopy classes of smooth arcs which start at one of the punctures and end on the boundary of the disk, considered up to sliding the endpoint along the boundary of $D_4$. We require that these isotopy classes have intersection number with every vertical line in the disk at most 1 and with the horizontal axis at least 1 (not counting the endpoints).
The bijective correspondence between the two notions is as follows: given such an arc $\gamma$ we consider a tubular neighborhood $N$ of $\gamma\cup\partial D_4$; the corresponding simple closed curve is then the boundary of $D_n-N$.

\begin{figure}[htb]
\centerline{\input{redcurves.pstex_t}}
\caption{}\label{curve3-2}
\end{figure}

Figure \ref{curve3-2} gives all the possibilities of such arcs. 
%
We shall first make two remarks concerning this figure which will be useful for the end of this section:
\begin{remark}
\label{remarquesurlescourbes}
(1) Curves of type 1 and 2 are symmetric to each other with respect to the horizontal axis. Hence acting by $\sigma_i$ on one of these curves and by $\sigma_i^{-1}$ on the other yields two curves with the same symmetry. The curves of type 3 and 4 have the same properties with respect to each other.

(2) One can obtain the curves 3 from the curves 2 (and 4 from 1) by applying the half twist $\Delta_4$. Moreover this symmetry is preserved by the respective actions of a braid $x$ and its conjugate $\tau(x)$.
\end{remark}

\subsection{The two braids $x$ and $\hat x$ have the same canonical length}
For a 4-braid~$y$ with an essential reduction curve surrounding 3 punctures let us denote by $v_y$ the number of crossings in $y$ where the outer strand is involved, counted with sign. This number is invariant under conjugacy. We remark that the winding number of the outer strand with each other strand is $\frac{v_y}{6}$.

\begin{lemma}
\label{lemma1}
Let $y$ be a positive reducible 4-braid with an essential reduction curve of complexity 2 surrounding 3 punctures. Then $v_y<3\sup(y)$.
\end{lemma}
\begin{proof}
Because $y$ is a positive braid, $y$ is a product of $\sup(y)$ simple braids. Because each simple braid contributes at most 3 to $v_y$, we have $v_y\leqslant 3\sup(y)$.
Suppose that the equality holds, then each simple braid contributes exactly 3 to $v_y$,
so that the outer strand crosses all the inner strands in each Garside factor of $y$.
If the outer strand is the first or the fourth then $\sup(y)$ is even because of the purity of this outer strand, and the circle surrounding the punctures 2, 3, 4, or 1, 2, 3, respectively, is preserved by $y$; this is a contradiction since both of these circles cross each of the curves in Figure \ref{curve3-2}.\\
If the outer strand is the second or the third then the Garside factors of $y$ other than~$\Delta$ must be $\sigma_1\sigma_2\sigma_1\sigma_3\sigma_2$ and $\sigma_2\sigma_3\sigma_2\sigma_1\sigma_2$ alternately, so that by the purity of the outer strand, $\sup(y)$ is also even and the circle surrounding the punctures numbered 1 and 2 and the circle surrounding the punctures numbered 3 and~4 are preserved by $y$. We conclude as above.
\end{proof}

 We now see that removing the outer strand of $x$ does not affect the canonical length:
 \begin{proposition}
 \label{infsup}
The equalities $\inf(\hat x)=\inf(x)$ and ${\sup(\hat x)=\sup(x)}$ hold.
\end{proposition}

\begin{proof}
First notice that $\sup(\hat x)\leqslant \sup(x)$ and $\inf(\hat x)\geqslant \inf(x)$ because $\hat x$ was obtained from $x$ by removing a strand.
Let $P$ be the minimal standardizer of $\mathcal C$. By removing the outer strand in~${P^{-1}x P}$ we obtain a 3-braid $x'$, conjugate to $\hat x$: if $P'\in B_3$ is obtained from $P$ by removing the outer strand, then $x'=P'^{-1}\hat x P'$. Since $P^{-1}x P$ preserves a round curve surrounding 3 punctures, its left normal form can be written (see \cite{leelee}) as 
$$P^{-1}x P=\langle x_0\rangle x_1,$$ 
with $x_0\in B_2$ and $x_1\in B_3$. (This notation means that $x_1$ is the braid obtained from $P^{-1}xP$ by removing the outer strand and $x_0$ is the 2-braid obtained when considering $x_1$ as a fat strand.) Note that $x_1=x'$. Hence if $P''$ is the 4-braid obtained from $P'$ by adding a trivial strand in the suitable position and if we define~$X=P''P^{-1}x PP''^{-1}$, then $X=\langle x_0\rangle \hat x$; that is $X$ is a conjugate of $x$ in which we have a tubular braid equal to $\hat x$ and an outer strand making $\frac{v_x}{6}$ twists around this ``fat strand" (we also have 
$x_0=\sigma_1^{v_x/3}$).

Now according to \cite{leelee} one has $$\sup(X)=\max(\sup(x_0),\sup(\hat x))$$ and because~$x\in SSS(x)$ one has $\sup(X)\geqslant \sup(x)$.

Next we prove that $\sup(x_0)<\sup(X)$ (and in particular $\sup(X)=\sup(\hat x)$). For if we had $\sup(X)=\sup(x_0)$, then $v_X=3\sup(X)$ because of the roundness of the curve preserved by $X$; and by conjugacy $v_x=v_X$.  Now, according to Lemma \ref{lemma1} we also have ${v_x<3\sup(x)}$. Hence $$v_x<3\sup(x)\leqslant 3\sup(X)=v_X=v_x,$$ which is a contradiction. The second part of the proposition follows as $$\sup(\hat x)\leqslant\sup(x)\leqslant\sup(X)=\sup(\hat x).$$
In a similar way we now have $$\inf(X)=\min(\inf(x_0),\inf(\hat x)).$$ Moreover $\inf(X)<\inf(x_0)$,  for if we had 
an equality, $\Delta^{-\inf(X)}X$ would be a split braid (see Section 6 in \cite{leelee}).
Noticing that (by purity of the outer strand) $\inf(X)$ is even, it would follow that $\Delta^{-\inf(X)}x$ is also split and positive (because $x\in SSS(x)$, one has $\inf(X)\leqslant \inf(x)$) so that by Proposition 6.2 in \cite{leelee} the outermost curves in $CRS(\Delta^{-\inf(X)}x)=CRS(x)$ are round.
This is impossible since the only circles which do not cross the curves in Figure \ref{curve3-2} are inner to them. Hence $$\inf(X)=\inf(\hat x)\geqslant\inf(x)\geqslant\inf(X),$$ where the last inequality holds since $x\in SSS(x)$. We finally obtain $$\inf(X)=\inf(\hat x)=\inf(x).$$ Hence the proposition is shown and we remark also that $X\in SSS(x)$.
\end{proof}

\begin{corollary}
With the above notations we have $\hat x\in SSS(\hat x)$.
\end{corollary}
\begin{proof}
Let us suppose by contradiction that there exists a 3-braid $\hat z$ in the conjugacy class of $\hat x$ with $\inf(\hat z)>\inf(\hat x)$ or $\sup(\hat z)<\sup(\hat x)$. Let us also denote by~$\hat y$ the conjugating element, that is $\hat z=\hat y^{-1}\hat x\hat y$. Let $z$ be the 4-braid obtained by conjugating $X$ by $\hat y$ augmented with a trivial strand in the suitable position, such that~$z$ has the same round essential reduction curve as $X$ and $z=\langle x_0\rangle\hat z$.

By the same argument using Lemma \ref{lemma1} as in the proof of Proposition \ref{infsup} (with~$z$ playing the role of $X$), we have $\sup(z)=\sup(\hat z)$ and $\inf(z)=\inf(\hat z)$. 

Now suppose that $\sup(\hat z)<\sup(\hat x)$. Then $$\sup(z)=\sup(\hat z)<\sup(\hat x)=\sup(x);$$ which is a contradiction, because the braids $z$ and $x$ are conjugate and $x$ lies in its own $SSS$.

Similarly suppose that $\inf(\hat z)>\inf(\hat x)$. Then $\inf(z)=\inf(\hat z)>\inf(\hat x)=\inf(x)$. This is again a contradiction.
\end{proof}
\begin{corollary}
The braid $\hat x$ is rigid.
\end{corollary}
\begin{proof}

See Proposition \ref{rigid}.
\end{proof}

\subsection{Analyzing the left normal form of $x$} 
 Our strategy for proving Proposition~\ref{3punctures} is to obtain a contradiction by proving the following statement:
\begin{lemma}\label{technicalemma}
None of the curves in Figure \ref{curve3-2} can be an essential reduction curve for $x$.
\end{lemma}
We do this by analyzing precisely the factors of the left normal form of $x$. These are composed of an inner 3-braid and possibly a move of the outer strand.  Because of the equalities $\sup(x)=\sup(\hat x)$ and $\inf(x)=\inf(\hat x)$ (see Proposition~\ref{infsup}), and according to Proposition \ref{formenormale3}, the factors of the left normal form of $\hat x$ are exactly the inner components of the factors of the left normal form of $x$. 
Recall that the only nontrivial non-$\Delta$ simple elements in $B_3$ are 
$\sigma_1$, $\sigma_2$, $\sigma_1\sigma_2$ and~$\sigma_2\sigma_1$.

\begin{table}[htb]
$$\begin{tabular}{|c|c|c|}
\hline 
\backslashbox{Inner factor}{Number of the outer strand}
  & 1 & 2\\
  \hline
                                           $\sigma_1$ &

  \begin{tabular}{c}\label{Tablefactors}  
      $\sigma_2$\\
      $\sigma_2\sigma_1$\\
      $\sigma_2\sigma_1\sigma_2$\\
      $\sigma_2\sigma_1\sigma_2\sigma_3$
      \end{tabular} &

 \begin{tabular}{c}
    $\sigma_1\sigma_2$\\
     $\sigma_1\sigma_2\sigma_1$\\
     $\sigma_2\sigma_1$\\
     $\sigma_2\sigma_1\sigma_3$
     \end{tabular}\\
     
     \hline
     
     $\sigma_1\sigma_2$ &

 \begin{tabular}{c}
      $\sigma_2\sigma_3$\\
      $\sigma_2\sigma_1\sigma_3$\\
      $\sigma_2\sigma_1\sigma_3\sigma_2$\\
      $\sigma_2\sigma_1\sigma_3\sigma_2\sigma_3$
      \end{tabular}
      &

    \begin{tabular}{c}
       
          $\sigma_1\sigma_2\sigma_3$\\
     $\sigma_1\sigma_2\sigma_1\sigma_3$\\
     $\sigma_2\sigma_1\sigma_3\sigma_2$\\
     $\sigma_2\sigma_1\sigma_3\sigma_2\sigma_3$
     \end{tabular}\\
\hline
$ \sigma_2 $    &
   \begin{tabular}{c}    
          $\sigma_3$\\
       $\sigma_3\sigma_1$\\
       $\sigma_1\sigma_2\sigma_3\sigma_2$\\
       $\sigma_3\sigma_1\sigma_2$
       
       \end{tabular}
       & 
       \begin{tabular}{c} 
       $\sigma_3$\\
       $\sigma_3\sigma_1$\\
       $\sigma_3\sigma_2$\\
       $\sigma_3\sigma_2\sigma_3$

\end{tabular}\\
\hline
$\sigma_2\sigma_1$ & 

\begin{tabular}{c}
       $\sigma_3\sigma_2\sigma_1$\\
       $\sigma_3\sigma_1\sigma_2\sigma_1$\\
       $\sigma_3\sigma_2$\\
       $\sigma_1\sigma_2\sigma_3\sigma_2\sigma_1$
       \end{tabular} &

      \begin{tabular}{c}
       $\sigma_3\sigma_2\sigma_1$\\
       $\sigma_3\sigma_1\sigma_2$\\
       $\sigma_3\sigma_2\sigma_3\sigma_1$\\
       $\sigma_3\sigma_1\sigma_2\sigma_1$
\end{tabular}\\

\hline
\end{tabular}$$
\caption{\vspace{-4mm}}
\label{Tablefactors}
\end{table}

Table \ref{Tablefactors} gives all the possibilities of non-trivial non-$\Delta$ positive simple braids in~$B_4$ depending on the position of the outer strand (at the beginning of the braid) and on the value of the inner braid; for every possibility on the inner braid there are four possibilities for the corresponding braid on four strands. For the rest of this section we shall abbreviate ``nontrivial non-$\Delta$ simple 4-braids" by ``simple 4-braids" and they will be supposed positive unless otherwise stated.

We saw that curves for which the outer strand is the third or the fourth are images under the half-twist $\Delta$ of curves for which the outer strand is the second or the first, respectively. If the outer strand is the third (or the fourth) then the simple 4-braids whose interior component is $\sigma_1$, $\sigma_1\sigma_2$, $\sigma_2$, or $\sigma_2\sigma_1$, are images under the automorphism $\tau$ of simple 4-braids with outer strand in second position (or first position respectively) and whose inner component is $\sigma_2 $, $\sigma_2\sigma_1$, $\sigma_1$, or $\sigma_1\sigma_2$, respectively. This allows us to construct the rest of Table~\ref{Tablefactors}.

\subsection{Proof of Lemma \ref{technicalemma}}

We make a constant implicit use of Proposition \ref{gebhardt}, asserting that if $\Delta^px_1\cdots x_r$ is the left normal form of $x$ and if $x$ preserves one of the complexity 2 curves, then for $1\leqslant i\leqslant r$ the image of the considered curve under~${\Delta^px_1\cdots x_i}$ is again a complexity~2 (or lower than 2) curve.
By the following four lemmas we are going to eliminate all types of curves depicted in Figure \ref{curve3-2}.

\begin{lemma}
\label{t}
The curves of type $\mathcal T$ cannot be essential reduction curves of $x$.
\end{lemma}
\begin{proof}
First notice that it is sufficient to prove the claim for curves $\mathcal T_1$ and $\mathcal T_2$ (because of the symmetries between the curves).

For both curves $\mathcal T_1$ and $\mathcal T_2$ the outer strand is the first. The simple~4-braids whose outer strand is the first and whose inner component starts with the letter~$\sigma_2$ or~$\sigma_1$ send the curve $\mathcal T_1$ or $\mathcal T_2$, respectively, to strictly more complex curves. By the symmetries mentioned above, the simple 4-braids whose outer strand is the fourth and whose inner component starts with the letter~$\sigma_1$ or $\sigma_2$ send the curves~$\mathcal T_4$ or~$\mathcal T_3$, respectively, to strictly more complex curves. Similarly, by Remark \ref{remarquesurlescourbes}, the action of a negative simple 4-braid whose outer strand is the first and whose inner component starts with the letter $\sigma_1^{-1}$ or~$\sigma_2^{-1}$ on the curve~$\mathcal T_1$ or $\mathcal T_2$, respectively, yields a strictly more complex curve (as does the action of a negative simple 4-braid whose outer strand is the fourth and whose inner component starts with the 
 letter~$\sigma_2^{-1}$ or $\sigma_1^{-1}$ on the curve $\mathcal T_4$ or $\mathcal T_3$, respectively).

Now suppose that the curve $\mathcal T_1$ or $\mathcal T_2$ is an essential reduction curve for $x$.

If $\inf(x)$ is even then the first factor of $\Delta^{-\inf(x)}\hat x$ must start with the letter $\sigma_1$ (or~$\sigma_2$, respectively), by Proposition \ref{gebhardt} and the previous paragraph. By the rigidity of~$\hat x$ the last factor of $\hat x$ (which is the inner component of the last factor of $x$) must end with the letter $\sigma_1$ (or $\sigma_2$, respectively). Because of the symmetries between the curves $\mathcal T_1$ and $\mathcal T_2$ mentioned above, the image of the essential reduction curve at the beginning of the last factor of $x$ has to be of complexity greater than 2. This contradicts Proposition \ref{gebhardt}.

If $\inf(x)$ is odd, then the essential reduction curve at the beginning of the first non-$\Delta$ factor of $x$ is $\mathcal T_4$ or $\mathcal T_3$, respectively.  Thus the first letter of $\Delta_3^{-\inf(x)}\hat x$ must be $\sigma_2$, or $\sigma_1$, respectively. And by the rigidity of $\hat x$ the last letter of $\hat x$ must then be $\sigma_1$, or $\sigma_2$, respectively. As before we obtain a contradiction by considering the symmetries with respect to the horizontal axis.

These two cases and the remark at the beginning of the proof together achieve the proof of the lemma.
\end{proof}


\begin{lemma}
\label{u}
The curves of type $\mathcal U$ cannot be essential reduction curves for $x$.
\end{lemma}
\begin{proof}
As before we prove the statement only for curves $\mathcal U_1$ and $\mathcal U_2$.
We are going to assume, for a contradiction, that one of the curves $\mathcal U_1$ or $\mathcal U_2$ is an essential reduction curve for $x$. Hence Proposition \ref{gebhardt} forces the image of this curve under each Garside factor of $x$ to be of complexity at most 2.
\begin{claim}
Assume that the curve $\mathcal U_1$ is an essential reduction curve for $x$. Then the first letter of $\Delta_3^{-\inf(x)}\hat x$ is $$\begin{cases} \sigma_2 & \text{if \,\  $\inf(x)$ \,\ is even},  \\
\sigma_1 & \text{if \,\ $\inf(x)$ \,\ is odd.}
\end{cases}$$
\end{claim}
\noindent\textit{Proof.} Figure \ref{diagram}(a) is a guide for the proof of this claim. We are just going to use Proposition \ref{gebhardt} and the left-weightedness of pairs of consecutive factors of $\hat x$.

First suppose that $\inf(x)$ is even. In order to show that the first letter of $\Delta^{-\inf(x)}\hat x$ is $\sigma_2$, we are going to search which simple 4-braids could be the factors of the left normal form of $\Delta^{-\inf(x)}x$, provided the first letter of $\Delta^{-\inf(x)}\hat x$ is~$\sigma_1$.
Among all possible simple 4-braids whose outer strand is the second and whose inner component starts with the letter $\sigma_1$, only two do not increase the complexity of the curve~$\mathcal U_1$. These two are $\sigma_1\sigma_2$ and $\sigma_1\sigma_2\sigma_3$. We prove that neither of these two braids can be the first factor of $\Delta^{-\inf(x)}x$.

The first sends $\mathcal U_1$ to $\mathcal S_1$ and the second to $\mathcal T_1$; their inner components are $\sigma_1$ and~$\sigma_1\sigma_2$, respectively. Now $\mathcal T_1$ is sent to strictly more complex curves by simple four-braids whose inner component starts with the letter $\sigma_2$ (by Lemma \ref{t}), thus the only possibility left is the first.

Now the only simple 4-braids whose inner component starts with the letter $\sigma_1$, whose outer strand is the first (as in the curve $\mathcal S_1$) and which do not increase the complexity of the curve $\mathcal S_1$ are $\sigma_2$ and $\sigma_2\sigma_3$.
 The first fixes this curve and its inner braid is $\sigma_1$, the second sends $\mathcal S_1$ to $\mathcal T_1$ and its inner braid is $\sigma_1\sigma_2$; by the same argument as above this second case is impossible.
 
 Hence we have shown that if the first letter of $\Delta^{-\inf(x)}\hat x$ is $\sigma_1$ then the curve~$\mathcal U_1$ cannot be preserved by $x$. This shows the first part of the claim.
 
\indent Now suppose that $\inf(x)$ is odd.
The essential reduction curve $\mathcal U_1$ of $x$ is transformed (after the action of an odd number of factors $\Delta$) into the curve $\mathcal U_4$ at the beginning of the first factor of $\Delta^{-\inf(x)}x$. This situation is analogous to the situation we described in the proof of the first half of the claim, up to applying $\tau$. This proves the second part of Claim 1.

\begin{figure}[htb]
\centerline{\input{diagramme1courbesu.pstex_t}}
\caption{(a) The action on the curve $\mathcal U_1$ of simple 4-braids whose inner component starts with the letter $\sigma_1$.
(b) The action on the curve $\mathcal U_2$ of simple 4-braids whose inner component starts with the letter $\sigma_2$. Underlined letters $\sigma_i$ indicate inner braids. Bold crosses indicate curves of complexity greater than 2.
We can also see in (a) the action on the curve $\mathcal S_1$ of simple 4-braids whose inner component starts with the letter $\sigma_1$ and in (b) the action on the curve $\mathcal S_2$ of simple 4-braids whose inner component starts with the letter $\sigma_2$. This will be used in Lemma \ref{s}}
\label{diagram}\end{figure}
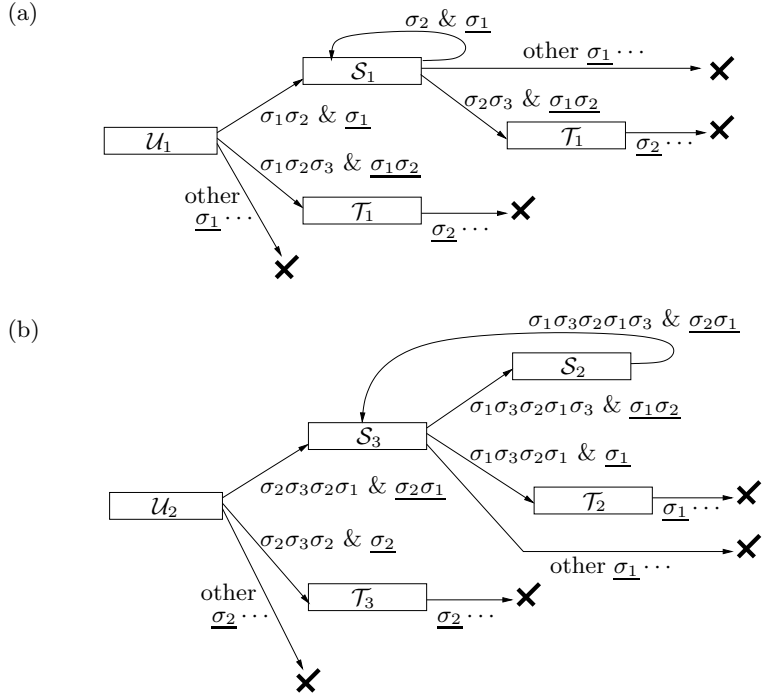

 \begin{claim}
Assume that the curve $\mathcal U_2$ is an essential reduction curve for $x$. Then the first letter of $\Delta_3^{-\inf(x)}\hat x$ is $$\begin{cases} 
\sigma_1 & \text{if \,\ $\inf(x)$ \,\ is even},  \\
\sigma_2 & \text{if \,\  $\inf(x)$ \,\  is odd.}
\end{cases}$$
\end{claim}

\noindent\textit{Proof.} This proof is illustrated in Figure \ref{diagram}(b). Again we look at which simple~${\text{4-braids}}$ can occur in the left normal form of $\Delta^{-\inf(x)}x$.

First suppose that $\inf(x)$ is even. Among all simple 4-braids whose interior braid starts with the letter $\sigma_2$ and whose outer strand is the second, only two do not increase the complexity of the curve $\mathcal U_2$, namely $\sigma_2\sigma_3\sigma_2$ and $\sigma_2\sigma_3\sigma_2\sigma_1$. The first one has as its inner braid $\sigma_2$ and it sends $\mathcal U_2$ to $\mathcal T_3$ whereas $\mathcal T_3$ is sent to strictly more complex curves by all simple 4-braids whose inner component starts with the letter~$\sigma_2$: this case cannot occur. The second one induces $\sigma_2\sigma_1$ on the inner strands and sends $\mathcal U_2$ to $\mathcal S_3$.

Among all suitable simple 4-braids the only ones which do not increase the complexity of the curve $\mathcal S_3$ are $\sigma_1\sigma_3\sigma_2\sigma_1$ and $\sigma_1\sigma_3\sigma_2\sigma_1\sigma_3$. Their inner components are~$\sigma_1$ and $\sigma_1\sigma_2$, and they send the curve $\mathcal S_3$ to $\mathcal T_2$ and $\mathcal S_2$, respectively. Because simple 4-braids whose inner component starts with the letter $\sigma_1$ always send $\mathcal T_2$ to strictly more complex curves, the only possibility left is the second. The following Garside factor of $x$ is then preceded by the curve $\mathcal S_2$ and its inner braid starts with the letter $\sigma_2$. This situation is the image under $\tau $ of the situation at the beginning of the preceding factor. Thus this third factor must be $\sigma_1\sigma_3\sigma_2\sigma_1\sigma_3$ and at its end the curve becomes~$\mathcal S_3$, a situation we already treated. Hence we see that if the inner braid $\Delta^{-\inf(x)}\hat x$ 
 starts with the letter $\sigma_2$, then we never retrieve the curve $\mathcal U_2$. This shows the first part of the claim.

\indent Now, suppose that $\inf(x)$ is odd.
In this situation, the odd number of factors equal to $\Delta$ transform $\mathcal U_2$ to $\mathcal U_3$. Up to applying $\tau$ we saw that no 3-braid starting with the letter $\sigma_1$ can be the inner braid of $\Delta^{-\inf(x)}x$, hence $\Delta^{-\inf(x)}x$ has to start with the letter $\sigma_2$.
This achieves the proof~of~Claim 2.

 Now, since $\hat x$ is rigid, if $\mathcal U_1$ is preserved by $x$ then the last letter of $\hat x$ must be~$\sigma_2$. However using the symmetry with respect to the horizontal axis, the action of negative simple 4-braids whose inner component starts with the letter $\sigma_2^{-1}$ on the curve~$\mathcal U_1$ can be seen from the action of positive simple 4-braids whose inner component starts with the letter $\sigma_2$ on the curve $\mathcal U_2$. So we look at the action of the reversed word rev($\Delta^{-\inf(x)}x$) on the curve $\mathcal U_2$. Though this word, with separations as in the left normal form of $x$, is not necessarily in left normal form, the word~rev($\Delta^{-\inf(x)}\hat x$) is, by Proposition \ref{formenormale3}. We saw in the proof of Claim 2 that this action cannot yield the curves $\mathcal U_1$ or $\mathcal U_2$. (Notice that in the proofs of the two claims we only used the left-weightedness of the pairs of consecutive factors in~$\hat x$ and the fact that the image of the essential reduction curve is of complexity bounded by~2 after each Garside factor of $x$.)
 Similarly if $\mathcal U_2$ is preserved by~$x$ then the last letter of $\hat x$ must be $\sigma_1$. An argument analogous to the previous one yields the desired contradiction.
\end{proof}

\begin{lemma}
\label{s}
The curves of type $\mathcal S$ cannot be essential reduction curves for $x$.
\end{lemma}
\begin{proof}
The proof is similar to the proof of the previous lemma and can be derived with the help of Figure \ref{diagram}.
We just have to use an additional argument since in fact there exist both: factors whose inner component starts with the letter~$\sigma_1$ and which preserve the complexity of the curve $\mathcal S_1$ on one hand, and factors whose inner component starts with the letter $\sigma_2$ and which preserve the complexity of the curve~$\mathcal S_2$ on the other hand. However, in both cases these factors are unique and $\Delta^{-\inf(x)}x$ only consists of repetitions of the following factors: $\sigma_2$ in the first case,~{yielding} a circle preserved by $x$ and which crosses the curves of type $\mathcal S$ (namely the circle surrounding punctures 2, 3 and 4); $\sigma_1\sigma_3\sigma_2\sigma_1\sigma_3$ in the second case, also yielding circles having non-empty intersection with the curves of type~$\mathcal S$ and being preserved by $x$ (the two circles surrounding punctures 1,2 and 3 and punctures~2,3 and~4).
At this point we have achieved the analogues of Claims 1 and 2 of the previous lemma; we conclude in the same way because of the rigidity of $\hat x$.
\end{proof}

\begin{lemma}
\label{r}
The curves of type $\mathcal R$ cannot be essential reduction curves for $x$.
\end{lemma}

\begin{proof}
The proof is modeled on that of Lemma \ref{u}. We prove the statement only for curves $\mathcal R_1$ and $\mathcal R_2$.
 \begin{claim}
Assume that the curve $\mathcal R_1$ is an essential reduction curve for $x$. Then the first letter of $\Delta_3^{-\inf(x)}\hat x$ is $$\begin{cases} 
\sigma_1 & \text{if \,\ $\inf(x)$ \,\ is even},  \\
\sigma_2 & \text{if \,\  $\inf(x)$ \,\  is odd.}
\end{cases}$$
\end{claim}
\noindent\textit{Proof.} See Figure \ref{diagram2}(a).

We first suppose that $\inf(x)$ is even. The only simple 4-braids whose outer strand is numbered 1, whose inner component starts with the letter $\sigma_2$, and which do not increase the complexity of the curve $\mathcal R_1$ when applied to it, are $\sigma_3\sigma_2\sigma_1, \sigma_3\sigma_2$ and~$\sigma_3$. The first two have $\sigma_2\sigma_1$ as their inner braid, the third has~$\sigma_2$. Since~${\mathcal R_1*\sigma_3\sigma_2\sigma_1=\mathcal U_1}$, the following Garside factor of $x$ (whose inner component has to start with the letter~$\sigma_1$) must be~$\sigma_1\sigma_2$ or~$\sigma_1\sigma_2\sigma_3$ as in the proof of Lemma~\ref{u}. This implies that in this case we can never retrieve the curve~$\mathcal R_1$, nor obtain the curve $\mathcal R_4$.

We also have $\mathcal R_1*\sigma_3\sigma_2=\mathcal T_2$. Because of Lemma \ref{t} we can eliminate this case. Finally~$\sigma_3$ (which has $\sigma_2$ 
as its inner braid) fixes the curve $\mathcal R_1$. Thus the only possibility left, provided the first letter of $\Delta^{-\inf(x)}\hat x$ is~$\sigma_2$, is that $x=\Delta^{\inf(x)}\sigma_3^p$ for some natural integer~$p$; in this case $x$ preserves a round curve (the circle surrounding punctures~${2, 3 \text{ and } 4}$) whose intersection with $\mathcal R_1$ is non-empty.

\begin{figure}[hbt]
\centerline{\input{diagramme2.pstex_t}}
\caption{(a) The action on the curve $\mathcal R_1$ of simple 4-braids whose inner component starts with the letter $\sigma_2$.
(b) The action on the curve $\mathcal R_2$ of simple 4-braids whose inner component starts with the letter $\sigma_1$. Underlined letters $\sigma_i$ indicate inner factors. Bold crosses indicate curves of complexity greater than 2. Factors which immediately yield more complex curves are not represented here.}\label{diagram2}
\end{figure}
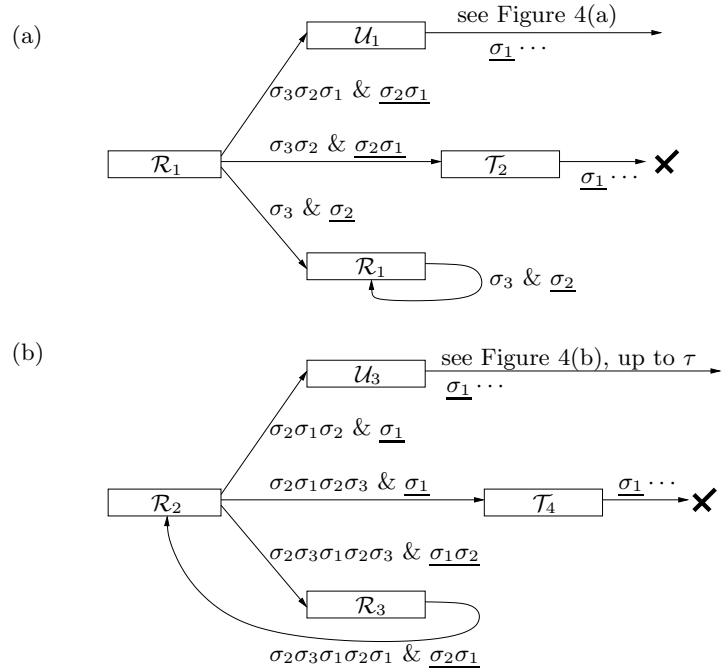

Up to applying $\tau$ (for the case when $\inf(x)$ is odd) this also shows the second statement of the claim.

\begin{claim}
Assume that the curve $\mathcal R_2$ is an essential reduction curve for $x$. Then the first letter of $\Delta_3^{-\inf(x)}\hat x$ is $$\begin{cases} 
\sigma_2 & \text{if \,\ $\inf(x)$ \,\ is even},  \\
\sigma_1 & \text{if \,\  $\inf(x)$ \,\  is odd.}
\end{cases}$$
\end{claim}

\noindent\textit{Proof.} See Figure \ref{diagram2}(b).
We prove the claim only for even $\inf$, the case of odd $\inf$ is symmetric.
The only simple 4-braids whose inner component starts with the letter~$\sigma_1$ and which do not increase the complexity of the curve $\mathcal R_2$ are ${\sigma_2\sigma_1\sigma_2,\,\sigma_2\sigma_1\sigma_2\sigma_3}$ and~$\sigma_2\sigma_3\sigma_1\sigma_2\sigma_3$. The first two induce $\sigma_1$ on the inner strands. They yield $\mathcal U_3$ and~$\mathcal T_4$, respectively. By Lemma \ref{t} we can exclude the latter. The first case is the image under $\Delta$ of the situation we described in Lemma \ref{u}. In this case we can never obtain $\mathcal R_2$ or $\mathcal R_3$.

Finally  the only possibility left is $\sigma_2\sigma_3\sigma_1\sigma_2\sigma_3$ (which induces $\sigma_1\sigma_2$ as its inner braid) and one has $\mathcal R_2*\sigma_2\sigma_3\sigma_1\sigma_2\sigma_3=\mathcal R_3$. According to the discussion just above, the following Garside factor of $x$ must be $\tau(\sigma_2\sigma_3\sigma_1\sigma_2\sigma_3)= \sigma_2\sigma_1\sigma_3\sigma_2\sigma_1$ (since~${\mathcal R_3=\mathcal R_2*\Delta}$). This factor has $\sigma_2\sigma_1$ as its inner component and yields the curve $\mathcal R_2$ when applied to the curve $\mathcal R_3$. This forces the braid $\Delta^{-\inf(x)}x$ to consist of a succession of the factors  $\sigma_2\sigma_3\sigma_1\sigma_2\sigma_3$ and  $\sigma_2\sigma_1\sigma_3\sigma_2\sigma_1$ alternately. In this case~$x$ preserves round curves (namely the two circles surrounding the punctures 1 and 2 and the punctures~3 and 4) which cross the curve $\mathcal R_2$.
This achieves the proof of the claim.

Claims 3 and 4 determine the last letter of $\hat x$ because of the rigidity of $\hat x$, and we conclude exactly as we did in the previous lemmas.
\end{proof}

These last four lemmas imply Lemma \ref{technicalemma}; since we had an exhaustive description of curves of complexity 2 in $D_4$ surrounding 3 punctures, Proposition \ref{3punctures} is shown.

\section{Examples and conjectures}

In this last section we shall see that our result is sharp in some sense. We give some examples which we obtained with the aid of \cite{braiding}.

The first example shows that we cannot remove the case of almost round curves in the statement of Proposition \ref{3punctures}.

\begin{example}\rm
Let us consider the braid $$x=\sigma_1\sigma_2\sigma_3\sigma_2.\sigma_2\sigma_1\sigma_3.\sigma_3\sigma_1.\sigma_3\sigma_2\sigma_1.\sigma_1\in B_4,$$ which is in left normal form as written.
Then $x\in SSS(x)$ and $x$ has an almost round essential reduction curve surrounding 3 punctures and no round reduction curve. See Figure \ref{example1}(a).
\end{example}

There are similar examples for curves surrounding two punctures:
\begin{example}\rm
Consider the 4-braid $$y=\sigma_2\sigma_3\sigma_1. \sigma_1\sigma_2\sigma_3\sigma_2. \sigma_2. \sigma_2\sigma_3.$$ Again $y$ lies in its own $SSS$, has no round reduction curve, and preserves an almost round curve surrounding 2 punctures.
See Figure \ref{example1}(b).
\end{example}

\begin{figure}[hbt]
\centerline{\input{example1modified.pstex_t}}
\caption{The braids $x$ and $y$ and their essential reduction curves. The dashed lines separate Garside factors.}\label{example1}
\end{figure}
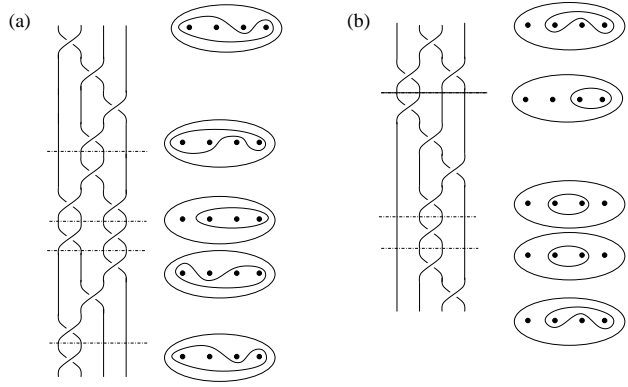

Now let us show that our result cannot be extended to the $n\geqslant 5$ case. 
\begin{example}\rm\label{exempletrois}
Consider the braid $$z=\sigma_1\sigma_2\sigma_1\sigma_3\sigma_2\sigma_1.\sigma_1\sigma_3.\sigma_1\sigma_3\sigma_2\sigma_4.\sigma_2\sigma_1\sigma_4\sigma_3\sigma_2\sigma_1.\sigma_1\sigma_2\sigma_1.\sigma_1\sigma_2\sigma_1\sigma_3\sigma_2.\sigma_3\in B_5,$$ which is in left normal form as written. 
Then $z \in SSS(z)$ and $z$ has neither round reduction curves, nor almost round reduction curves. In fact $z$ preserves the complexity 2 curve shown in Figure \ref{examples}(a). Notice that the interior braid is pseudo-Anosov.
\end{example}

We conjecture that our results can be improved:
\begin{conjecture}
Proposition \ref{3punctures} can be generalized to every simple closed curve in~$D_4$ (even those surrounding 2 punctures) using the same kind of arguments.
\end{conjecture}
However there are more curves of complexity 2 surrounding 2 punctures in $D_4$ than curves of complexity 2 surrounding three punctures in $D_4$, thus the proof would be more difficult.

We finish by briefly looking at two conceivable alternative approaches to the reducibility problem in $B_n$, for every $n\in~\mathbb{N}$. The first one concerns the cyclic sliding operation:
\begin{conjecture}\label{conjectureslide} There is a polynomial bound on the number of times one has to apply cyclic sliding in order to decrease the complexity of any essential reduction curve which is not round or almost round.
\end{conjecture}

\begin{figure}[hbt]
\centerline{\input{examples2345.pstex_t}}
\caption{}
\label{examples}
\end{figure}

Conjecture~\ref{conjectureslide} would imply a polynomial algorithm for solving the reducibility problem in all braid groups. This is because the results in~\cite{jingtao} imply that the complexity of any essential reduction curve of a reducible braid is linearly bounded by the length of this braid. The following example shows that the most optimistic version of Conjecture \ref{conjectureslide}, namely that the complexity of the reduction curve decreases in every single step, is wrong.
\begin{example}\rm
Consider the braid $z$ in Example \ref{exempletrois}. After one cyclic sliding one obtains from~$z$ the braid $$\mathfrak{s}(z)=\sigma_1\sigma_3.\sigma_1\sigma_3\sigma_2\sigma_1\sigma_4.\sigma_2\sigma_4\sigma_3\sigma_2\sigma_1.\sigma_1\sigma_2\sigma_3\sigma_2\sigma_1.\sigma_1\sigma_2\sigma_3\sigma_2.\sigma_2\sigma_1\sigma_3\sigma_2\sigma_1.\sigma_1$$ in left normal form, which still has a curve of complexity 2, shown in Figure \ref{examples}(b), as an essential reduction curve.
\end{example}

Another approach to the 
reducibility problem in $B_n$ comes from the following:
\begin{conjecture}
\label{conj2}
Let $x=\Delta^px_1\cdots x_r$ be a reducible braid in left normal form
such that $x\in SSS(x)$, and let~$\mathcal C$ be an essential reduction 
curve of~$x$. Then there is some~$i$ between~1 and~$r$ such that the 
curve~$\mathcal C$ is sent to a round or almost round curve by the braid 
$\Delta^px_1\cdots x_i$. 
\end{conjecture}
Note that Theorem \ref{bkl}, together with the truth of this conjecture,
would yield a polynomial time algorithm for solving the reducibility
problem in~$B_n$.

The hypothesis that $x\in SSS(x)$ in the statement of Conjecture~\ref{conj2} 
is necessary, as the following example shows:
\begin{example}\rm
Let $u=\sigma_1\sigma_2\sigma_1\sigma_3\sigma_2.\sigma_2\sigma_3\sigma_2\sigma_1.\sigma_1.\sigma_1\sigma_2\sigma_3.\sigma_3\sigma_2\in B_4$. The essential reduction curve of $u$ is of complexity 2 and it is sent to curves of complexity 2 after each Garside factor of $u$ (see Figure \ref{examples}(c)). Notice that the inner braid is also pseudo-Anosov.
\end{example}

{\bf Acknowledgements } The authors are grateful to Sergey Matveev for helpful comments on the paper, and in particular for bringing to  their attention Corollary~\ref{ito}. The first-named author's doctoral studies are supported by a grant ARED (R\'egion Bretagne). He was also supported by a grant for international mobility from the Universit\'e Europ\'eenne de Bretagne.



\end{document}

%% file: figurecourbered.pstex_t
\begin{picture}(0,0)%
\includegraphics{figurecourbered.pstex}%
\end{picture}%
\setlength{\unitlength}{2693sp}%
\begingroup\makeatletter\ifx\SetFigFont\undefined%
\gdef\SetFigFont#1#2#3#4#5{%
  \reset@font\fontsize{#1}{#2pt}%
  \fontfamily{#3}\fontseries{#4}\fontshape{#5}%
  \selectfont}%
\fi\endgroup%
\begin{picture}(6046,1620)(-104,-916)
\end{picture}%

%% file: rond.pstex_t
\begin{picture}(0,0)%
\includegraphics{rond.pstex}%
\end{picture}%
\setlength{\unitlength}{1865sp}%
\begingroup\makeatletter\ifx\SetFigFont\undefined%
\gdef\SetFigFont#1#2#3#4#5{%
  \reset@font\fontsize{#1}{#2pt}%
  \fontfamily{#3}\fontseries{#4}\fontshape{#5}%
  \selectfont}%
\fi\endgroup%
\begin{picture}(7022,3420)(135,-2896)
\end{picture}%

%% file: redcurves.pstex_t
\begin{picture}(0,0)%
\includegraphics{redcurves.pstex}%
\end{picture}%
\setlength{\unitlength}{2155sp}%
\begingroup\makeatletter\ifx\SetFigFont\undefined%
\gdef\SetFigFont#1#2#3#4#5{%
  \reset@font\fontsize{#1}{#2pt}%
  \fontfamily{#3}\fontseries{#4}\fontshape{#5}%
  \selectfont}%
\fi\endgroup%
\begin{picture}(8628,6915)(-914,-6241)
\put(-899,-196){\makebox(0,0)[lb]{\smash{{\SetFigFont{9}{10.8}{\rmdefault}{\bfdefault}{\updefault}{\color[rgb]{0,0,0}$\mathcal R$}%
}}}}
\put(-899,-1951){\makebox(0,0)[lb]{\smash{{\SetFigFont{9}{10.8}{\rmdefault}{\bfdefault}{\updefault}{\color[rgb]{0,0,0}$\mathcal S$}%
}}}}
\put(-899,-3796){\makebox(0,0)[lb]{\smash{{\SetFigFont{9}{10.8}{\rmdefault}{\bfdefault}{\updefault}{\color[rgb]{0,0,0}$\mathcal T$}%
}}}}
\put(-899,-5551){\makebox(0,0)[lb]{\smash{{\SetFigFont{9}{10.8}{\rmdefault}{\bfdefault}{\updefault}{\color[rgb]{0,0,0}$\mathcal U$}%
}}}}
\end{picture}%

%% file: diagramme1courbesu.pstex_t
\begin{picture}(0,0)%
\includegraphics{diagramme1courbesu.pstex}%
\end{picture}%
\setlength{\unitlength}{1776sp}%
\begingroup\makeatletter\ifx\SetFigFont\undefined%
\gdef\SetFigFont#1#2#3#4#5{%
  \reset@font\fontsize{#1}{#2pt}%
  \fontfamily{#3}\fontseries{#4}\fontshape{#5}%
  \selectfont}%
\fi\endgroup%
\begin{picture}(10569,9672)(661,-10690)
\put(8701,-8086){\makebox(0,0)[lb]{\smash{{\SetFigFont{9}{10.8}{\rmdefault}{\mddefault}{\updefault}{\color[rgb]{0,0,0}$\mathcal T_2$}%
}}}}
\put(6601,-4261){\makebox(0,0)[lb]{\smash{{\SetFigFont{9}{10.8}{\rmdefault}{\mddefault}{\updefault}{\color[rgb]{0,0,0}$\underline{\sigma_2}\cdots$}%
}}}}
\put(676,-1261){\makebox(0,0)[lb]{\smash{{\SetFigFont{9}{10.8}{\rmdefault}{\mddefault}{\updefault}{\color[rgb]{0,0,0}(a)}%
}}}}
\put(676,-5686){\makebox(0,0)[lb]{\smash{{\SetFigFont{9}{10.8}{\rmdefault}{\mddefault}{\updefault}{\color[rgb]{0,0,0}(b)}%
}}}}
\put(3151,-3811){\makebox(0,0)[lb]{\smash{{\SetFigFont{9}{10.8}{\rmdefault}{\mddefault}{\updefault}{\color[rgb]{0,0,0}other}%
}}}}
\put(3301,-4111){\makebox(0,0)[lb]{\smash{{\SetFigFont{9}{10.8}{\rmdefault}{\mddefault}{\updefault}{\color[rgb]{0,0,0}$\underline{\sigma_1}\cdots$}%
}}}}
\put(4201,-3361){\makebox(0,0)[lb]{\smash{{\SetFigFont{9}{10.8}{\rmdefault}{\mddefault}{\updefault}{\color[rgb]{0,0,0}$\sigma_1\sigma_2\sigma_3\ \&\ \underline{\sigma_1}\underline{\sigma_2}$}%
}}}}
\put(7051,-2461){\makebox(0,0)[lb]{\smash{{\SetFigFont{9}{10.8}{\rmdefault}{\mddefault}{\updefault}{\color[rgb]{0,0,0}$\sigma_2\sigma_3\ \&\ \underline{\sigma_1}\underline{\sigma_2}$}%
}}}}
\put(7876,-1786){\makebox(0,0)[lb]{\smash{{\SetFigFont{9}{10.8}{\rmdefault}{\mddefault}{\updefault}{\color[rgb]{0,0,0}other $\underline{\sigma_1}\cdots$}%
}}}}
\put(6226,-1336){\makebox(0,0)[lb]{\smash{{\SetFigFont{9}{10.8}{\rmdefault}{\mddefault}{\updefault}{\color[rgb]{0,0,0}$\sigma_2\ \&\ \underline{\sigma_1}$}%
}}}}
\put(3376,-9361){\makebox(0,0)[lb]{\smash{{\SetFigFont{9}{10.8}{\rmdefault}{\mddefault}{\updefault}{\color[rgb]{0,0,0}other}%
}}}}
\put(3526,-9661){\makebox(0,0)[lb]{\smash{{\SetFigFont{9}{10.8}{\rmdefault}{\mddefault}{\updefault}{\color[rgb]{0,0,0}$\underline{\sigma_2}\cdots$}%
}}}}
\put(7126,-6736){\makebox(0,0)[lb]{\smash{{\SetFigFont{9}{10.8}{\rmdefault}{\mddefault}{\updefault}{\color[rgb]{0,0,0}$\sigma_1\sigma_3\sigma_2\sigma_1\sigma_3\ \&\ \underline{\sigma_1\sigma_2}$}%
}}}}
\put(4201,-7861){\makebox(0,0)[lb]{\smash{{\SetFigFont{9}{10.8}{\rmdefault}{\mddefault}{\updefault}{\color[rgb]{0,0,0}$\sigma_2\sigma_3\sigma_2\sigma_1\ \&\ \underline{\sigma_2\sigma_1}$}%
}}}}
\put(9826,-8161){\makebox(0,0)[lb]{\smash{{\SetFigFont{9}{10.8}{\rmdefault}{\mddefault}{\updefault}{\color[rgb]{0,0,0}$\underline{\sigma_1}\cdots$}%
}}}}
\put(8251,-8986){\makebox(0,0)[lb]{\smash{{\SetFigFont{9}{10.8}{\rmdefault}{\mddefault}{\updefault}{\color[rgb]{0,0,0}other $\underline{\sigma_1}\cdots$}%
}}}}
\put(7951,-5536){\makebox(0,0)[lb]{\smash{{\SetFigFont{9}{10.8}{\rmdefault}{\mddefault}{\updefault}{\color[rgb]{0,0,0}$\sigma_1\sigma_3\sigma_2\sigma_1\sigma_3\ \&\ \underline{\sigma_2\sigma_1}$}%
}}}}
\put(4201,-8611){\makebox(0,0)[lb]{\smash{{\SetFigFont{9}{10.8}{\rmdefault}{\mddefault}{\updefault}{\color[rgb]{0,0,0}$\sigma_2\sigma_3\sigma_2\ \&\ \underline{\sigma_2}$}%
}}}}
\put(7126,-7411){\makebox(0,0)[lb]{\smash{{\SetFigFont{9}{10.8}{\rmdefault}{\mddefault}{\updefault}{\color[rgb]{0,0,0}$\sigma_1\sigma_3\sigma_2\sigma_1\ \&\ \underline{\sigma_1}$}%
}}}}
\put(4201,-2686){\makebox(0,0)[lb]{\smash{{\SetFigFont{9}{10.8}{\rmdefault}{\mddefault}{\updefault}{\color[rgb]{0,0,0}$\sigma_1\sigma_2\ \&\ \underline{\sigma_1}$}%
}}}}
\put(2626,-3061){\makebox(0,0)[lb]{\smash{{\SetFigFont{9}{10.8}{\rmdefault}{\mddefault}{\updefault}{\color[rgb]{0,0,0}$\mathcal U_1$}%
}}}}
\put(5476,-2086){\makebox(0,0)[lb]{\smash{{\SetFigFont{9}{10.8}{\rmdefault}{\mddefault}{\updefault}{\color[rgb]{0,0,0}$\mathcal S_1$}%
}}}}
\put(9451,-3061){\makebox(0,0)[lb]{\smash{{\SetFigFont{9}{10.8}{\rmdefault}{\mddefault}{\updefault}{\color[rgb]{0,0,0}$\underline{\sigma_2}\cdots$}%
}}}}
\put(8401,-2986){\makebox(0,0)[lb]{\smash{{\SetFigFont{9}{10.8}{\rmdefault}{\mddefault}{\updefault}{\color[rgb]{0,0,0}$\mathcal T_1$}%
}}}}
\put(5476,-4036){\makebox(0,0)[lb]{\smash{{\SetFigFont{9}{10.8}{\rmdefault}{\mddefault}{\updefault}{\color[rgb]{0,0,0}$\mathcal T_1$}%
}}}}
\put(8401,-6211){\makebox(0,0)[lb]{\smash{{\SetFigFont{9}{10.8}{\rmdefault}{\mddefault}{\updefault}{\color[rgb]{0,0,0}$\mathcal S_2$}%
}}}}
\put(5551,-7186){\makebox(0,0)[lb]{\smash{{\SetFigFont{9}{10.8}{\rmdefault}{\mddefault}{\updefault}{\color[rgb]{0,0,0}$\mathcal S_3$}%
}}}}
\put(2701,-8161){\makebox(0,0)[lb]{\smash{{\SetFigFont{9}{10.8}{\rmdefault}{\mddefault}{\updefault}{\color[rgb]{0,0,0}$\mathcal U_2$}%
}}}}
\put(6676,-9661){\makebox(0,0)[lb]{\smash{{\SetFigFont{9}{10.8}{\rmdefault}{\mddefault}{\updefault}{\color[rgb]{0,0,0}$\underline{\sigma_2}\cdots$}%
}}}}
\put(5476,-9436){\makebox(0,0)[lb]{\smash{{\SetFigFont{9}{10.8}{\rmdefault}{\mddefault}{\updefault}{\color[rgb]{0,0,0}$\mathcal T_3$}%
}}}}
\end{picture}%

%% file: diagramme2.pstex_t
\begin{picture}(0,0)%
\includegraphics{diagramme2.pstex}%
\end{picture}%
\setlength{\unitlength}{1776sp}%
\begingroup\makeatletter\ifx\SetFigFont\undefined%
\gdef\SetFigFont#1#2#3#4#5{%
  \reset@font\fontsize{#1}{#2pt}%
  \fontfamily{#3}\fontseries{#4}\fontshape{#5}%
  \selectfont}%
\fi\endgroup%
\begin{picture}(9927,9289)(661,-9995)
\put(4276,-9886){\makebox(0,0)[lb]{\smash{{\SetFigFont{9}{10.8}{\rmdefault}{\mddefault}{\updefault}{\color[rgb]{0,0,0}$\sigma_2\sigma_3\sigma_1\sigma_2\sigma_1\ \&\ \underline{\sigma_2\sigma_1}$}%
}}}}
\put(676,-1261){\makebox(0,0)[lb]{\smash{{\SetFigFont{9}{10.8}{\rmdefault}{\mddefault}{\updefault}{\color[rgb]{0,0,0}(a)}%
}}}}
\put(676,-5686){\makebox(0,0)[lb]{\smash{{\SetFigFont{9}{10.8}{\rmdefault}{\mddefault}{\updefault}{\color[rgb]{0,0,0}(b)}%
}}}}
\put(7351,-1411){\makebox(0,0)[lb]{\smash{{\SetFigFont{9}{10.8}{\rmdefault}{\mddefault}{\updefault}{\color[rgb]{0,0,0}$\underline{\sigma_1}\cdots$}%
}}}}
\put(8626,-3211){\makebox(0,0)[lb]{\smash{{\SetFigFont{9}{10.8}{\rmdefault}{\mddefault}{\updefault}{\color[rgb]{0,0,0}$\underline{\sigma_1}\cdots$}%
}}}}
\put(6901,-961){\makebox(0,0)[lb]{\smash{{\SetFigFont{9}{10.8}{\rmdefault}{\mddefault}{\updefault}{\color[rgb]{0,0,0}see Figure 4(a)}%
}}}}
\put(4276,-3661){\makebox(0,0)[lb]{\smash{{\SetFigFont{9}{10.8}{\rmdefault}{\mddefault}{\updefault}{\color[rgb]{0,0,0}$\sigma_3\ \&\ \underline{\sigma_2}$}%
}}}}
\put(7351,-4636){\makebox(0,0)[lb]{\smash{{\SetFigFont{9}{10.8}{\rmdefault}{\mddefault}{\updefault}{\color[rgb]{0,0,0}$\sigma_3 \ \&\ \underline{\sigma_2}$}%
}}}}
\put(4276,-2761){\makebox(0,0)[lb]{\smash{{\SetFigFont{9}{10.8}{\rmdefault}{\mddefault}{\updefault}{\color[rgb]{0,0,0}$\sigma_3\sigma_2\ \&\ \underline{\sigma_2\sigma_1}$}%
}}}}
\put(4276,-2011){\makebox(0,0)[lb]{\smash{{\SetFigFont{9}{10.8}{\rmdefault}{\mddefault}{\updefault}{\color[rgb]{0,0,0}$\sigma_3\sigma_2\sigma_1\ \&\ \underline{\sigma_2\sigma_1}$}%
}}}}
\put(6751,-6136){\makebox(0,0)[lb]{\smash{{\SetFigFont{9}{10.8}{\rmdefault}{\mddefault}{\updefault}{\color[rgb]{0,0,0}$\underline{\sigma_1}\cdots$}%
}}}}
\put(6676,-5761){\makebox(0,0)[lb]{\smash{{\SetFigFont{9}{10.8}{\rmdefault}{\mddefault}{\updefault}{\color[rgb]{0,0,0}see Figure 4(b), up to $\tau$}%
}}}}
\put(9151,-7486){\makebox(0,0)[lb]{\smash{{\SetFigFont{9}{10.8}{\rmdefault}{\mddefault}{\updefault}{\color[rgb]{0,0,0}$\underline{\sigma_1}\cdots$}%
}}}}
\put(4276,-8461){\makebox(0,0)[lb]{\smash{{\SetFigFont{9}{10.8}{\rmdefault}{\mddefault}{\updefault}{\color[rgb]{0,0,0}$\sigma_2\sigma_3\sigma_1\sigma_2\sigma_3\ \&\ \underline{\sigma_1\sigma_2}$}%
}}}}
\put(4276,-6736){\makebox(0,0)[lb]{\smash{{\SetFigFont{9}{10.8}{\rmdefault}{\mddefault}{\updefault}{\color[rgb]{0,0,0}$\sigma_2\sigma_1\sigma_2\ \&\ \underline{\sigma_1}$}%
}}}}
\put(4276,-7486){\makebox(0,0)[lb]{\smash{{\SetFigFont{9}{10.8}{\rmdefault}{\mddefault}{\updefault}{\color[rgb]{0,0,0}$\sigma_2\sigma_1\sigma_2\sigma_3\ \&\ \underline{\sigma_1}$}%
}}}}
\put(2626,-3061){\makebox(0,0)[lb]{\smash{{\SetFigFont{9}{10.8}{\rmdefault}{\mddefault}{\updefault}{\color[rgb]{0,0,0}$\mathcal R_1$}%
}}}}
\put(2626,-7786){\makebox(0,0)[lb]{\smash{{\SetFigFont{9}{10.8}{\rmdefault}{\mddefault}{\updefault}{\color[rgb]{0,0,0}$\mathcal R_2$}%
}}}}
\put(5476,-1261){\makebox(0,0)[lb]{\smash{{\SetFigFont{9}{10.8}{\rmdefault}{\mddefault}{\updefault}{\color[rgb]{0,0,0}$\mathcal U_1$}%
}}}}
\put(7276,-3061){\makebox(0,0)[lb]{\smash{{\SetFigFont{9}{10.8}{\rmdefault}{\mddefault}{\updefault}{\color[rgb]{0,0,0}$\mathcal T_2$}%
}}}}
\put(5476,-4486){\makebox(0,0)[lb]{\smash{{\SetFigFont{9}{10.8}{\rmdefault}{\mddefault}{\updefault}{\color[rgb]{0,0,0}$\mathcal R_1$}%
}}}}
\put(5476,-5986){\makebox(0,0)[lb]{\smash{{\SetFigFont{9}{10.8}{\rmdefault}{\mddefault}{\updefault}{\color[rgb]{0,0,0}$\mathcal U_3$}%
}}}}
\put(7951,-7786){\makebox(0,0)[lb]{\smash{{\SetFigFont{9}{10.8}{\rmdefault}{\mddefault}{\updefault}{\color[rgb]{0,0,0}$\mathcal T_4$}%
}}}}
\put(5476,-9211){\makebox(0,0)[lb]{\smash{{\SetFigFont{9}{10.8}{\rmdefault}{\mddefault}{\updefault}{\color[rgb]{0,0,0}$\mathcal R_3$}%
}}}}
\end{picture}%

%% file: example1modified.pstex_t
\begin{picture}(0,0)%
\includegraphics{example1modified.pstex}%
\end{picture}%
\setlength{\unitlength}{1243sp}%
\begingroup\makeatletter\ifx\SetFigFont\undefined%
\gdef\SetFigFont#1#2#3#4#5{%
  \reset@font\fontsize{#1}{#2pt}%
  \fontfamily{#3}\fontseries{#4}\fontshape{#5}%
  \selectfont}%
\fi\endgroup%
\begin{picture}(12314,7603)(121,-6929)
\end{picture}%

%% file: examples2345.pstex_t
\begin{picture}(0,0)%
\includegraphics{examples2345.pstex}%
\end{picture}%
\setlength{\unitlength}{1184sp}%
\begingroup\makeatletter\ifx\SetFigFont\undefined%
\gdef\SetFigFont#1#2#3#4#5{%
  \reset@font\fontsize{#1}{#2pt}%
  \fontfamily{#3}\fontseries{#4}\fontshape{#5}%
  \selectfont}%
\fi\endgroup%
\begin{picture}(17966,15170)(5311,-25281)
\end{picture}%